\documentclass[12pt]{article}
\usepackage{amsfonts,amssymb,amsmath,amsthm}
\date{}

\theoremstyle{plain}
\newtheorem{theorem}{Theorem}[section]
\newtheorem{proposition}{Proposition}[section]
\newtheorem{lemma}{Lemma}[section]
\newtheorem{corollary}{Corollary}[section]
\theoremstyle{definition}
\newtheorem{definition}{Definition}[section]

\newtheorem{example}{Example}[section]
\numberwithin{equation}{section}
\numberwithin{theorem}{section}
\numberwithin{proposition}{section}
\numberwithin{lemma}{section}
\numberwithin{corollary}{section}
\numberwithin{definition}{section}

\newcommand{\R}{\mathbb{R}}
\newcommand{\N}{\mathbb{N}}
\newcommand{\Z}{\mathbb{Z}}
\newcommand{\Q}{\mathbb{Q}}
\newcommand{\D}{\mathcal{D}}
\newcommand{\T}{\mathbb{T}}
\newcommand{\esslim}{\operatornamewithlimits{ess\,lim}}

\newcommand{\essinf}{\operatornamewithlimits{ess\,inf}}
\newcommand{\esssup}{\operatornamewithlimits{ess\,sup}}
\newcommand{\sgn}{\operatorname{sign}}

\newcommand{\const}{\mathrm{const}}
\renewcommand{\div}{\operatorname{div}}
\newcommand{\supp}{\operatorname{supp}}
\newcommand{\meas}{\operatorname{meas}}
\newcommand{\Tr}{\operatorname{Tr}}
\newcommand{\Cl}{\operatorname{Cl}}
\newcommand{\Int}{\operatorname{Int}}
\title{On some properties of entropy solutions of degenerate non-linear anisotropic parabolic equations}
\author{Evgeny Yu. Panov\footnote{Novgorod State University, 41, B.St-Petersburgskaia str., 173003 Veliky Novgorod, Russian Federation and Peoples’ Friendship University of Russia (RUDN University),
6 Miklukho-Maklaya St, Moscow, 117198, Russian Federation}}
\begin{document}
\maketitle
\begin{abstract}
We prove existence of the largest and the smallest entropy solutions to the Cauchy problem for a nonlinear degenerate anisotropic parabolic equation. Applying this result, we establish the comparison principle in the case when at least one of the initial functions is periodic. In the case when initial function vanishes at infinity (in the sense of strong average) we prove the long time decay of an entropy solution under exact nonlinearity-diffusivity condition.

\medskip
Keywords: nonlinear parabolic equations, conservation laws, Cauchy problem, entropy solutions, comparison principle,
nonlinearity-diffusivity condition, decay property
\end{abstract}

\maketitle

\section{Introduction}\label{Intro}
In the half-space $\Pi=\R_+\times\R^n$, $\R_+=(0,+\infty)$, we consider the nonlinear parabolic equation
\begin{equation}\label{1}
u_t+\div_x(\varphi(u)-a(u)\nabla_x u)=0,
\end{equation}
where the flux vector $\varphi(u)=(\varphi_1(u),\ldots,\varphi_n(u))$ is merely continuous: $\varphi_i(u)\in C(\R)$, $i=1,\ldots,n$, and the diffusion matrix $a(u)=(a_{ij}(u))_{i,j=1}^n$ is Lebesgue measurable and bounded:
$a_{ij}(u)\in L^\infty(\R)$, $i,j=1,\ldots,n$. We also assume that the matrix ${a(u)\ge 0}$
(nonnegative definite). This matrix may have nontrivial kernel. Hence (\ref{1}) is a degenerate
(hyperbolic-parabolic) equation. In particular case $a\equiv 0$ it reduces to a first order conservation law
 \begin{equation}\label{con}
u_t+\div_x \varphi(u)=0.
\end{equation}
Equation (\ref{1}) can be written (at least formally) in the conservative form
\begin{equation}\label{1c}
u_t+\div_x\varphi(u)-D^2_x\cdot A(u)=0,
\end{equation}
where the matrix $A(u)$ is a primitive of the matrix $a(u)$, $A'(u)=a(u)$, and the operator $D^2_x$ is the second order  ``divergence''
$$
D^2_x\cdot A(u)\doteq\sum_{i,j=1}^n \frac{\partial^2}{\partial x_i\partial x_j}A_{ij}(u), \quad u=u(t,x).
$$
Equation (\ref{1}) is endowed with the initial condition
\begin{equation}\label{2}
u(0,x)=u_0(x).
\end{equation}

Let $g(u)\in BV_{loc}(\R)$ be a function of bounded variation on any segment in $\R$.
We will need the bounded linear operator $T_g:C(\R)/C\to C(\R)/C$, where $C$ is the space of constants. This operator is defined up to an additive constant by the relation
\begin{equation}\label{efl}
T_g(f)(u)=g(u-)f(u)-\int_0^u f(s)dg(s),
\end{equation}
where
$\displaystyle g(u-)=\lim_{v\to u-} g(v)$ is the left limit of $g$ at the point $u$, and the integral in (\ref{efl}) is understood in accordance with the formula
$$
\int_0^u f(s)dg(s)=\sgn u\int_{J(u)} f(s)dg(s),
$$
where $\sgn u=1$,
$J(u)$ is the interval $[0,u)$ if $u>0$, and $\sgn u=-1$, $J(u)=[u,0)$ if $u\le 0$.
Observe that $T_g(f)(u)$ is continuous even in the case of discontinuous $g(u)$. For instance, if $g(u)=\sgn(u-k)$ then
$T_g(f)(u)=\sgn(u-k)(f(u)-f(k))$. Notice also that for $f\in C^1(\R)$ the operator $T_g$ is uniquely determined by the identity $T_g(f)'(u)=g(u)f'(u)$ (in $\D'(\R)$).

We fix some representation of the diffusion matrix $a(u)$ in the form $a(u)=b^\top(u)b(u)$, where $b(u)=(b_{ij}(u))_{i,j=1}^n$ is a matrix-valued function with measurable and bounded entries,  $b_{ij}(u)\in L^\infty(\R)$.
We recall the notion of entropy solution of the Cauchy problem (\ref{1}), (\ref{2}) introduced in \cite{ChPer1}.

\begin{definition}\label{def1}
A function $u=u(t,x)\in L^\infty(\Pi)$ is called an entropy solution (e.s. for short) of (\ref{1}), (\ref{2}) if
the following conditions hold:

(i) for each $j=1,\ldots,n$ the distributions
\begin{equation}\label{pr}
\div_x B_j(u(t,x))\in L^2_{loc}(\Pi),
\end{equation}
where vectors $B_j(u)=(B_{j1}(u),\ldots,B_{jn}(u))\in C(\R,\R^n)$, and $B_{ji}'(u)=b_{ji}(u)$, $j,i=1,\ldots,n$;

(ii) for every $g(u)\in C^1(\R)$, $j=1,\ldots,n$
\begin{equation}\label{cr}
\div_x T_g(B_j)(u(t,x))=g(u(t,x))\div_x B_j(u(t,x)) \ \mbox{ in } \D'(\Pi);
\end{equation}

(iii) for any convex function $\eta(u)\in C^2(\R)$
\begin{equation}\label{entr}
\eta(u)_t+\div_x T_{\eta'}(\varphi)(u)-D^2_x\cdot T_{\eta'}(A)(u)+\eta''(u)\sum_{j=1}^n (\div_x B_j(u))^2\le 0 \ \mbox{ in } \D'(\Pi);
\end{equation}

(iv) $\displaystyle\esslim_{t\to 0} u(t,\cdot)=u_0$ in $L^1_{loc}(\R^n)$.
\end{definition}

Relation (\ref{entr}) means that for any non-negative test function $f=f(t,x)\in C_0^\infty(\Pi)$
\begin{equation}\label{entrI}
\int_\Pi \bigl[\eta(u)f_t+T_{\eta'}(\varphi)(u)\cdot\nabla_x f+T_{\eta'}(A)(u)\cdot D^2_x f-f\eta''(u)\sum_{j=1}^n (\div_x B_j(u))^2\bigr]dtdx\ge 0,
\end{equation}
where $D^2_x f$ is the symmetric matrix of second order derivatives of $f$, and "$\cdot$" denotes the standard scalar multiplications of vectors or matrices (in particular, $A\cdot B=\Tr A^\top B$ for matrices $A,B$).

In the case of conservation laws (\ref{con}) Definition~\ref{def1} reduces to the known definition of entropy solutions
in the sense of S.\,N.~Kruzhkov \cite{Kr}. Taking in (\ref{entr}) $\eta(u)=\pm u$, we deduce that
$$
u_t+\div_x \varphi(u)-D^2_x\cdot A(u)=0 \ \mbox{ in } \D'(\Pi),
$$
that is, an e.s. $u$ is a weak solution of (\ref{1c}).

The main results of the paper are contained in the following four theorems.

\begin{theorem}\label{th2}
There exist the unique largest e.s. $u_+(t,x)$ and the smallest e.s. $u_-(t,x)$ of the problem (\ref{1}), (\ref{2}).  Moreover, $u_-(t,x)\le u_+(t,x)$.
\end{theorem}
The largest and the smallest e.s. exhibit $L^1$-contraction property. More precisely, the following result holds.

\begin{theorem}\label{th2contr}
Let $u_{1+}$, $u_{2+}$ be the largest e.s. of (\ref{1}), (\ref{2}) with initial data $u_{10}$, $u_{20}$. Then for a.e. $t>0$
$$
\int_{\R^n}(u_{1+}(t,x)-u_{2+}(t,x))^+dx\le\int_{\R^n}(u_{10}(x)-u_{20}(x))^+dx,
$$
where we use the standard notation $z^+=\max(z,0)$. The analogous property is valid the smallest e.s. $u_{1-}$ and $u_{2-}$: for a.e. $t>0$
$$
\int_{\R^n}(u_{1-}(t,x)-u_{2-}(t,x))^+dx\le\int_{\R^n}(u_{10}(x)-u_{20}(x))^+dx.
$$
\end{theorem}

On the base of Theorem~\ref{th2} we also establish the following \textit{comparison principle}.

\begin{theorem}\label{th4}
Let functions $u=u(t,x)$, $v=v(t,x)$ be e.s. of (\ref{1}), (\ref{2}) with corresponding initial data $u_0(x)$, $v_0(x)$, and $u_0(x)\le v_0(x)$. If at least one of the initial functions is periodic then $u(t,x)\le v(t,x)$ a.e. in $\Pi$.
\end{theorem}

The last our result is about the large time decay property of e.s. in the case when initial function vanishes at infinity in the sense of measure. More precisely we assume that
\begin{equation}\label{van}
\forall\lambda>0 \quad \meas \{ \ x\in\R^n: \ |u_0(x)|>\lambda \ \}<+\infty
\end{equation}
(in particular, this requirement is satisfied if $u_0(x)\to 0$ as $x\to\infty$ and also if
$u_0\in L^p(\R^n)$, $0<p<\infty$). It is not difficult to verify that (\ref{van}) is equivalent to the condition
that $|u_0|$ has the null strong mean value
\begin{equation}\label{SM}
\lim_{|I|\to\infty}\frac{1}{|I|}\int_I|u_0(x)|dx=0,
\end{equation}
where $I$ runs over Lebesgue measurable subsets of $\R^n$ with finite measure $|I|=\meas I$.

In fact, assuming (\ref{van}), we denote $A_\lambda=\{ \ x\in\R^n: \ |u_0(x)|>\lambda \ \}$, $\lambda>0$, so that $|A_\lambda|<+\infty$. Then for every measurable set $I$ of finite measure
$$
\int_I|u_0(x)|dx=\int_{I\setminus A_\lambda}|u_0(x)|dx+\int_{I\cap A_\lambda}|u_0(x)|dx\le \lambda|I|+\|u_0\|_\infty|A_\lambda|.
$$
Therefore,
$$
\limsup_{|I|\to\infty}\frac{1}{|I|}\int_I|u_0(x)|dx\le\lim_{|I|\to\infty}(\lambda+\|u_0\|_\infty|A_\lambda|/|I|)=\lambda.
$$
Since $\lambda>0$ is arbitrary, we deduce (\ref{SM}).
Conversely, suppose that (\ref{SM}) holds. We need to prove (\ref{van}). Assuming that (\ref{van}) is violating, we can find $\lambda>0$ such that the set $A_\lambda$ has infinite measure. Then we can choose the sequence of measurable subsets
$I_m\subset A_\lambda$ such that $|I_m|=m$, $m\in\N$. Obviously,
$$
\frac{1}{|I_m|}\int_{I_m}|u_0(x)|dx\ge\lambda
$$
while $|I_m|=m\to\infty$ as $m\to\infty$. But this contradicts (\ref{SM}), and we conclude that (\ref{van}) is satisfied.

\medskip
Assume also that the following \textit{nonlinearity-diffusivity} condition is satisfied
\begin{align}\label{GN}
\mbox{ On each interval } (a,0), (0,b) \mbox{ it cannot happen that the vector }\nonumber\\ \varphi(u) \mbox{ is affine and the matrix } a(u)\equiv 0 \mbox{ on this interval}.
\end{align}
In (\ref{GN}) and in the sequel the relation $p(u)\equiv 0$ on an interval, where $p(u)$ is a measurable function (or a vector function),  means that $a(u)=0$ almost everywhere on this interval.
Condition (\ref{GN}) simply means that our equation is not a first-order linear equation on each interval $u\in (a,0)$ or
$u\in (0,b)$.

To study the decay property, we introduce the topology on $L^\infty(\R^n)$ stronger than one induced by $L^1_{loc}(\R^n)$. This topology is generated by the following shift-invariant norm
\begin{equation}\label{normX}
\|u\|_X=\sup_{y\in\R^n} \int_{|x-y|<1} |u(x)|dx
\end{equation}
(where we denote by $|z|$ the Euclidean norm of a finite-dimensional vector $z$).
It is not difficult to verify that norm (\ref{normX}) is equivalent to each of more general norms
\begin{equation}\label{normV}
\|u\|_V=\sup_{y\in\R^n} \int_{y+V} |u(x)|dx,
\end{equation}
where $V$ is any bounded open set in $\R^n$ (the original norm $\|\cdot\|_X$ corresponds to the unit ball $|x|<1$). For the sake of completeness we prove this result in Lemma~\ref{lemA} below.

Our last result is the following decay property.

\begin{theorem}\label{thD}
If $u(t,x)$ is an e.s. of (\ref{1}), (\ref{2}) then, under assumptions (\ref{van}), (\ref{GN}),
\begin{equation}\label{dec}
\lim_{t\to+\infty}\|u(t,\cdot)\|_X=0.
\end{equation}
\end{theorem}

Notice that condition (\ref{GN}) is precise. In fact, if it fails, there exists an interval $(a,0)$ or $(0,b)$ where the vector $\varphi(u)$ is affine while the diffusion matrix $a(u)\equiv 0$. Assume for definiteness that
$\varphi(u)=uc+d$, $a(u)\equiv 0$ on an interval $(0,b)$, $b>0$, where $c,d\in\R^n$. If initial function $u_0(x)$ is such that $0\le u_0(x)\le b$ then an e.s. of (\ref{1}), (\ref{2}) is the traveling wave $u(t,x)=u_0(x-tc)$. If $u_0\not=0$ this e.s. does not satisfy (\ref{dec}).

Observe also that in the case when $u_0\ge 0$ ($u_0\le 0$) an e.s. $u\ge 0$ (respectively, $u\le 0$) as well, by the maximum/minimum principle (see Corollary~\ref{cor1b} below) and the nonlinearity-diffusivity assumption (\ref{GN}) in Theorem~\ref{thD} can be relaxed, it is sufficient to require that the flux vector $\varphi(u)$ is not affine or $a(u)\not\equiv 0$ on any interval $(0,b)$ (respectively, on any interval $(a,0)$).

Concerning exactness of condition (\ref{SM}) in Theorem~\ref{thD}, remark that we cannot replace it by a weaker condition
that $|u_0|$ has the null standard mean value, i.e.,
\begin{equation}\label{WM}
\lim_{R\to+\infty}R^{-n}\int_{|x|<R}|u_0(x)|dx=0.
\end{equation}
Let us confirm this by the following simple example.

\begin{example}
In the case of one space variable we consider the Cauchy problem for Burgers' equation
\begin{equation}\label{burg}
u_t+\frac{1}{2}(u^2)_x=0, \quad u(0,x)=u_0(x),
\end{equation}
where $u_0(x)$ is the indicator function of the set $\bigcup_{n=1}^\infty [2^n,2^n+n]$.
Since measure of this set is infinite, condition (\ref{SM}) is violated but condition (\ref{WM}) holds. In fact,
for $2^n\le R<2^{n+1}$
$$R^{-1}\int_{-R}^R|u_0(x)|dx\le 2^{-n}\sum_{k=1}^n k\le n^22^{-n}\mathop{\to}_{n\to\infty} 0.$$
Let $u(t,x)$ be the unique e.s. of (\ref{burg}). By the comparison principle $u(t,x)\ge u_n(t,x)$,
where $u_n(t,x)$ is an e.s. for Burgers' equation with initial data being the indicator function of the segment
$[2^n,2^n+n]$. It is clear that $u_n(t,x)=U(t/n,(x-2^n)/n)$, where
$U(t,y)$ is the e.s. of Burgers' equation $u_t+\frac{1}{2}(u^2)_y=0$, with initial data $u(0,y)=\chi_{[0,1]}(y)$
being the indicator of the segment $[0,1]$. This solution is well known. We only need the fact that $U(t,y)=1$ for
$t<y<1+t/2$. This implies that $u_n(t,y)=1$ if $2^n+t<y<2^n+n+t/2$. Since length of the interval $(2^n+t,2^n+n+t/2)$ is more than $1$ if $t<2n-2$, we obtain that $\|u(t,\cdot)\|_X\ge \|u_n\|_X\ge 1$ if $t<2n-2$. By arbitrariness of $n\in\N$ we conclude that $\|u(t,\cdot)\|_X\ge 1$ for all $t\ge 0$ and decay property (\ref{dec}) is violated.
\end{example}

\section{Preliminaries}

It will be convenient to include the initial condition in the integral entropy inequality.
\begin{proposition}\label{pro1}
A function $u=u(t,x)\in L^\infty(\Pi)$ satisfying conditions (i), (ii) is an e.s. of (\ref{1}), (\ref{2})
if and only if for any convex function $\eta(u)\in C^2(\R)$ and each nonnegative test function $f=f(t,x)\in C_0^\infty(\bar\Pi)$, where $\bar\Pi=[0,+\infty)\times\R^n$,
\begin{align}\label{eint}
\int_\Pi \bigl[\eta(u)f_t+T_{\eta'}(\varphi)(u)\cdot\nabla_x f+T_{\eta'}(A)(u)\cdot D^2_x f-f\eta''(u)\sum_{j=1}^n (\div_x B_j(u))^2\bigr]dtdx \nonumber\\ + \int_{\R^n} \eta(u_0(x))f(0,x)dx\ge 0.
\end{align}
\end{proposition}
\begin{proof}
Let $E$ be a set of $t>0$ such that $(t,x)$ is a Lebesgue point of $u(t,x)$ for almost all $x\in\R^n$. It is rather well-known (see for example \cite[Lemma~1.2]{PaJHDE}) that $E$ is a set of full measure and $t\in E$ is a common Lebesgue point of the functions $t\to\int_{\R^n} u(t,x)b(x)dx$ for all $b(x)\in L^1(\R^n)$. Since every Lebesgue point of $u$ is also a Lebesgue point of $\eta(u)$ for an arbitrary function $\eta\in C(\R)$, we may replace $u$ in the above property by $\eta(u)$. We choose a function $\omega(s)\in C_0^\infty(\R)$, such that $\omega(s)\ge 0$, $\supp\omega\subset [0,1]$, $\int\omega(s)ds=1$, and define the sequences $\omega_r(s)=r\omega(rs)$, $\theta_r(s)=\int_{-\infty}^s\omega_r(\sigma)d\sigma=\int_{-\infty}^{rs}\omega(\sigma)d\sigma$, $r\in\N$. Obviously, the sequence $\omega_r(s)$ converges as $r\to\infty$ to the Dirac $\delta$-measure weakly in $\D'(\R)$ while the sequence $\theta_r(s)$ converges to the Heaviside function $\theta(s)$ pointwise and in $L^1_{loc}(\R)$. Notice that $0\le\theta_r(s)\le 1$.
We take $f=f(t,x)\in C_0^\infty(\bar\Pi)$, $f\ge 0$, and $t_0\in E$. Applying (\ref{entr}) to the nonnegative test function $\theta_r(t-t_0)f(t,x)\in C_0^\infty(\Pi)$, we arrive at the relation
\begin{align}\label{i1}
\int_{\Pi} \eta(u)\omega_r(t-t_0)fdtdx+ \int_\Pi \bigl[\eta(u)f_t+T_{\eta'}(\varphi)(u)\cdot\nabla_x f + \nonumber\\ T_{\eta'}(A)(u)\cdot D^2_x f-f\eta''(u)\sum_{j=1}^n (\div_x B_j(u))^2\bigr]\theta_r(t-t_0)dtdx\ge 0.
\end{align}
Since
$$
\int_{\Pi} \eta(u)\omega_r(t-t_0)fdtdx=\int_0^{+\infty}\left(\int_{\R^n} \eta(u(t,x))f(t,x)dx\right)\omega_r(t-t_0)dt
$$
and
$t_0$ is a Lebesgue point of the function $t\to\int_{\R^n} \eta(u(t,x))f(t,x)dx$, it follows from (\ref{i1}) in the limit as $r\to\infty$ that
\begin{align}\label{i2}
\int_{\R^n} \eta(u(t_0,x))f(t_0,x)dx+ \int_{(t_0,+\infty)\times\R^n}\bigl[\eta(u)f_t+T_{\eta'}(\varphi)(u)\cdot\nabla_x f + \nonumber\\ T_{\eta'}(A)(u)\cdot D^2_x f-f\eta''(u)\sum_{j=1}^n (\div_x B_j(u))^2\bigr]dtdx\ge 0.
\end{align}
Now we pass in (\ref{i2}) to the limit as $E\ni t_0\to 0$.
Since
$$|\eta(u(t,x))-\eta(u_0(x))|\le C|u(t,x)-u_0(x)|, \quad C=\const,$$
we obtain that
$$
\lim_{E\ni t_0\to 0}\int_{\R^n}\eta(u(t_0,x))dx=\int_{\R^n}\eta(u_0(x))f(0,x)dx,
$$
where we take into account initial condition (iv).
With the help of this relation, the desired inequality (\ref{eint}) follows from (\ref{i2}) in the limit as $E\ni t_0\to 0$.

\medskip
Conversely, assume that relation (\ref{eint}) holds. Taking in this relation a nonnegative test function $f\in C_0^\infty(\Pi)$, we obtain (\ref{entrI}), showing that entropy condition (iii) is satisfied.
It only remains to prove initial requirement (iv) from Definition~\ref{def1}.
We fix a nonnegative function $h(x)\in C_0^\infty(\R^n)$, and apply (\ref{eint}) to the test function $f=h(x)(1-\theta_r(t-t_0))$, where $t_0\in E$.
As a result, we obtain
\begin{align*}
\int_{\R^n} \eta(u_0(x))h(x)dx-\int_{\Pi} \eta(u(t,x))\omega_r(t-t_0)hdtdx +
\int_{(0,t_0+1/r)\times\R^n} \bigl[T_{\eta'}(\varphi)(u)\cdot\nabla h + \\ T_{\eta'}(A)(u)\cdot D^2 h\bigr](1-\theta_r(t-t_0))dtdx\ge \int_\Pi f\eta''(u)\sum_{j=1}^n (\div_x B_j(u))^2dtdx\ge 0.
\end{align*}
Passing in this relation to the limit as $r\to\infty$, we arrive at the relation
\begin{align*}
\int_{\R^n} \eta(u_0(x))h(x)dx-\int_{\R^n} \eta(u(t_0,x))h(x)dx + \\
\int_{(0,t_0)\times\R^n} \bigl[T_{\eta'}(\varphi)(u)\cdot\nabla h + T_{\eta'}(A)(u)\cdot D^2 h\bigr]dtdx\ge 0,
\end{align*}
which implies in the limit as $E\ni t_0\to 0$ that
\begin{equation}\label{i3}
\limsup_{E\ni t_0\to 0} \int_{\R^n}\eta(u(t_0,x))h(x)dx\le \int_{\R^n} \eta(u_0(x))h(x)dx.
\end{equation}
By the continuity arguments, (\ref{i3}) remains valid for all convex entropies $\eta\in C(\R)$ (in particular, for $\eta(u)=|u-k|$, $k\in\R$) and all nonnegative functions $h(x)\in L^1(\R^n)$. We fix $\varepsilon>0$. Since $u_0(x)\in L^\infty(\R^n)$, we can find a step function
$v(x)=\sum_{i=1}^m v_i\chi_{A_i}(x)$, where $v_i\in\R$, $\chi_{A_i}(x)$ are indicator functions of measurable sets $A_i\subset\R^n$, such that $\|u_0-v\|_\infty<\varepsilon$.
The sets $A_i$, $i=1,\ldots,m,$ are supposed to be disjoint. In view of (\ref{i3})
\begin{align}\label{i4}
\limsup_{E\ni t_0\to 0} \int_{\R^n}|u(t_0,x)-v(x)|h(x)dx=\limsup_{E\ni t_0\to 0} \sum_{i=1}^m \int_{\R^n}|u(t_0,x)-v_i|\chi_{A_i}(x)h(x)dx \le \nonumber\\
\sum_{i=1}^m \int_{\R^n}|u_0(x)-v_i|\chi_{A_i}(x)h(x)dx= \int_{\R^n} |u_0(x)-v(x)|h(x)dx\le\varepsilon\|h\|_1.
\end{align}
Since $$|u(t_0,x)-u_0(x)|\le |u(t_0,x)-v(x)|+|v(x)-u_0(x)|<|u(t_0,x)-v(x)|+\varepsilon,$$
it follows from (\ref{i4}) that
$$
\limsup_{E\ni t_0\to 0} \int_{\R^n}|u(t_0,x)-u_0(x)|h(x)dx\le 2\varepsilon\|h\|_1
$$
and in view of arbitrariness of $\varepsilon>0$, we conclude that
$$
\lim_{E\ni t_0\to 0} \int_{\R^n}|u(t_0,x)-u_0(x)|h(x)dx=0
$$
for all $h(x)\in L^1(\R^n)$. Obviously, this implies that
$$
\esslim_{t\to 0+} (u(t,x)-u_0(x))=0 \ \mbox{ in } L^1_{loc}(\R^n)
$$
and completes the proof.
\end{proof}

We will need some a priory estimates of entropy solutions.

\begin{proposition}\label{pro2}
If $u=u(t,x)$ is an e.s. of (\ref{1}), (\ref{2}) then $\forall k\in\R$
\begin{equation}\label{L1}
\int_{\R^n}(u(t,x)-k)^+dx\le \int_{\R^n}(u_0(x)-k)^+dx
\end{equation}
for a.e. $t>0$.
\end{proposition}
\begin{proof}
It follows from (\ref{entr}) that for each convex $\eta\in C^2(\R)$
\begin{equation}\label{en1-}
\eta(u)_t+\div_x T_{\eta'}(\varphi)(u)-D^2_x\cdot T_{\eta'}(A)(u)\le 0 \ \mbox{ in } \D'(\Pi).
\end{equation}
By the continuity (\ref{en1-}) remains valid for any convex $\eta$.

Let $M=\|u\|_\infty$. We observe that (\ref{L1}) is nontrivial only if $\int_{\R^n}(u_0(x)-k)^+dx<+\infty$, which will be assumed in the sequel. Let us consider firstly the case $k=0$.
We denote for $m\ge n$, $\delta>0$
$$
\alpha(s)=\min((s^+)^m,1), \ \beta(k)=\alpha(k/\delta), \ \eta(u)=\int_{-\infty}^u\beta(k)dk=
\left\{\begin{array}{lcr} 0 & , & u\le 0, \\ \frac{u^{m+1}}{(m+1)\delta^m} & , & 0<u\le\delta, \\ u-\frac{m\delta}{m+1} & , & u>\delta.\end{array}\right.
$$
In view of (\ref{en1-}) for $u=u(t,x)$
\begin{equation}\label{en1}
\eta(u)_t+\div_x\psi(u)-D^2_x\cdot H(u)\le 0 \ \mbox{ in } \D'(\Pi),
\end{equation}
where we denote
\begin{align*}
\psi(u)=T_{\beta}(\varphi)(u)=\int_0^u (\varphi(u)-\varphi(k))\beta'(k)dk\in C(\R,\R^n), \\
H(u)=T_{\beta}(A)(u)=\int_0^u (A(u)-A(k))\eta'(k)dk\in C(\R,\R^{n\times n}).
\end{align*}
Since the matrix $A'(u)=a(u)\ge 0$, we find that the matrix $A(u)-A(k)\ge 0$ (nonnegative definite) if $k\le u$. Therefore,
the matrix $H(u)\ge 0$ (notice that $H(u)=0$ if $u\le 0$).
We notice that for $|u|\le M$
$$
|\psi(u)|\le 2\max_{|u|\le M}|\varphi(u)|\int_0^u\beta'(k)dk=2\max_{|u|\le M}|\varphi(u)|\beta(u)
$$
(here and in the sequel we denote by $|v|$ the Euclidean norm of a finite-dimensional vector $v$), and analogously
$$
|H(u)|\le 2\max_{|u|\le M}|A(u)|\beta(u).
$$
These estimates imply that for each $\varepsilon>0$
$$
\frac{|\psi(u)|}{\eta(u)+\varepsilon}\le \frac{C_1\beta(u)}{\eta(u)+\varepsilon}, \quad \frac{|H(u)|}{\eta(u)+\varepsilon}\le \frac{C_2\beta(u)}{\eta(u)+\varepsilon},
$$
where $\displaystyle C_1=2\max_{|u|\le M}|\varphi(u)|$, $\displaystyle C_2=2\max_{|u|\le M}|A(u)|$.

Since $\beta(u)=1$ for $u>\delta$ the function $\displaystyle \omega(u)\doteq \frac{\beta(u)}{\eta(u)+\varepsilon}$ decreases on $[\delta,+\infty)$. This implies that
$$
\max \omega(u)=\max_{[0,\delta]} \omega(u)\le\max_{u>0}\frac{(u/\delta)^m}{\delta(u/\delta)^{m+1}/(m+1)+\varepsilon}=\max_{v=u/\delta>0}\frac{m+1}{\delta v+(m+1)\varepsilon v^{-m}}.
$$
By direct computations we find
$$
\min_{v>0}(\delta v+(m+1)\varepsilon v^{-m})=\frac{\delta(m+1)}{m}\left(\frac{m(m+1)\varepsilon}{\delta}\right)^{\frac{1}{m+1}}.
$$
Therefore,
$$
\omega(u)\le\frac{m}{\delta}\left(\frac{\delta}{m(m+1)}\right)^{\frac{1}{m+1}}\varepsilon^{-\frac{1}{m+1}}.
$$
Hence
\begin{equation}\label{p2}
\frac{|\psi(u)|}{\eta(u)+\varepsilon}\le C\varepsilon^{-\frac{1}{m+1}}, \quad \frac{|H(u)|}{\eta(u)+\varepsilon}\le C\varepsilon^{-\frac{1}{m+1}},
\end{equation}
where $$C=\max(C_1,C_2)\frac{m}{\delta}\left(\frac{\delta}{m(m+1)}\right)^{\frac{1}{m+1}}=\const.$$
In view of (\ref{en1}) for every $\varepsilon>0$
$$
(\eta(u)+\varepsilon)_t+\div_x \psi(u)-D^2_x\cdot H(u)\le 0 \ \mbox{ in } \D'(\Pi),
$$
that is, for each test function $f=f(t,x)\in C_0^\infty(\Pi)$, $f\ge 0$
\begin{equation}\label{p3}
\int_\Pi[(\eta(u)+\varepsilon)f_t+\psi(u)\cdot\nabla_x f+H(u)\cdot D^2_xf]dtdx\ge 0.
\end{equation}
We choose a nonstrictly decreasing function $\rho(r)\in C^\infty(\R)$ with the properties $\rho(r)=1$ for $r\le 0$, $\rho(r)=e^{-r}$ for $r\ge 1$, it is concave on $(-\infty,1/2]$ and is convex on $[1/2,+\infty)$ (so that $1/2$ is an inflection point of $\rho(r)$). Such a function satisfies the inequality
\begin{equation}\label{ro}
\rho''(r)\le c|\rho'(r)|=-c\rho'(r)
\end{equation}
for some positive constant $c$. In fact, $\rho''(r)\le 0\le |\rho'(r)|$ for $r<1/2$, $\rho''(r)=-\rho'(r)=e^{-r}$ for $r>1$
while on the segment $[1/2,1]$ we have $-\rho'(r)\ge -\rho'(1)=e^{-1}$ by the convexity of $\rho(r)$, and therefore
$\rho''(r)\le -c\rho'(r)$ where $c=e\max\limits_{1/2\le r\le 1}\rho''(r)\ge 1$. We conclude that (\ref{ro}) holds.

Now we take the test function in the form
$$
f(t,x)=\rho(N(t-t_0)+|x|-R)\chi(t),
$$
where $0<t_0<T$, $R>1$, the constant $N=N(\varepsilon)$ will be indicated later, and a nonnegative function $\chi(t)\in C_0^\infty((0,t_0))$.
Observe that $f=\chi(t)$ in a neighborhood $|x|<R$ of the set where $x=0$, which implies that $f(t,x)\in C_0^\infty(\Pi)$.
Since the function $f$ and all its derivatives are exponentially vanishes as $|x|\to\infty$, we may choose the function $f$ as a test function in (\ref{p3}). Observe that
\begin{align}\label{p4a}
f_t(t,x)=N\rho'(N(t-t_0)+|x|-R)\chi(t)+\rho(N(t-t_0)+|x|-R)\chi'(t), \\
\label{p4b}
\nabla_x f=\rho'(N(t-t_0)+|x|-R)\chi(t)\frac{x}{|x|}, \\
\label{p5}
D^2_x f=\left(\rho''(N(t-t_0)+|x|-R)\frac{x\otimes x}{|x|^2}+\rho'(N(t-t_0)+|x|-R)\frac{|x|^2E-x\otimes x}{|x|^3}\right)\chi(t),
\end{align}
where $E$ denotes the unit matrix.
In view of (\ref{p5}) for all $\xi\in\R^n$,
\begin{align}\label{op}
(D^2_x f)\xi\cdot\xi=\nonumber\\ \chi(t)
\left(\rho''(N(t-t_0)+|x|-R)\frac{(x\cdot\xi)^2}{|x|^2}+\rho'(N(t-t_0)+|x|-R)\frac{|x|^2|\xi|^2-(x\cdot\xi)^2}{|x|^3}\right)
\le \nonumber\\ -c\rho'(N(t-t_0)+|x|-R)\frac{(x\cdot\xi)^2}{|x|^2}\chi(t),
\end{align}
where we use that $\rho'(\cdots)(|x|^2|\xi|^2-(x\cdot\xi)^2)\le 0$. Relation (\ref{op}) implies that the matrix
$$-c\rho'(N(t-t_0)+|x|-R)\chi(t)M(x)-D^2_x f\ge 0$$ (nonnegative definite), where $\displaystyle M(x)=\frac{x\otimes x}{|x|^2}$. Using the known property that the scalar product $A\cdot B\ge 0$ whenever matrices $A,B\ge 0$, we find that
\begin{align}\label{p5a}
H(u)\cdot D^2_x f\le -c\rho'(N(t-t_0)+|x|-R)\chi(t) H(u)\cdot M(x)\le\nonumber\\ -c\rho'(N(t-t_0)+|x|-R)\chi(t)|H(u)|
\end{align}
(notice that $|M(x)|=1$).
It now follows from (\ref{p3}) with the help of (\ref{p4a}), (\ref{p4b}) and (\ref{p5a}) that
\begin{align}\label{p6}
\int_\Pi[(\eta(u)+\varepsilon)\chi'(t)\rho(N(t-t_0)+|x|-R)dtdx+ \nonumber\\
\int_\Pi[N(\eta(u)+\varepsilon)-|\psi(u)|-c|H(u)|]\rho'(N(t-t_0)+|x|-R)\chi(t)dtdx\ge 0.
\end{align}
Taking in (\ref{p6}) $N=C(c+1)\varepsilon^{-\frac{1}{m+1}}$, we find that $N(\eta(u)+\varepsilon)-|\psi(u)|-c|H(u)|\ge 0$
in view of (\ref{p2}). Since $\rho'(r)\le 0$, the last integral in (\ref{p6}) is nonpositive and (\ref{p6}) implies that
\begin{align*}
\int\left(\int_{\R^n} (\eta(u)+\varepsilon)\rho(N(t-t_0)+|x|-R)dx\right)\chi'(t)dt=\\
\int_\Pi(\eta(u)+\varepsilon)\chi'(t)\rho(N(t-t_0)+|x|-R)dtdx\ge 0 \quad \forall \chi(t)\in C_0^\infty((0,t_0)), \ \chi(t)\ge 0,
\end{align*}
This means that
\begin{equation}\label{p6a}
\frac{d}{dt}\int_{\R^n} (\eta(u)+\varepsilon)\rho(N(t-t_0)+|x|-R)dx\le 0 \ \mbox{ in } \D'((0,t_0)).
\end{equation}
Let $E$ be a set of $t>0$ of full measure defined in the proof of Proposition~\ref{pro1}. Then every
$t\in E$ is a Lebesgue point of the function $t\to \int_{\R^n} (\eta(u)+\varepsilon)\rho(N(t-t_0)+|x|-R)dx$. If $t_0\in E$, then it follows from (\ref{p6a}) that for all $t\in E$, $t<t_0$
$$
\int_{\R^n} (\eta(u(t_0,x))+\varepsilon)\rho(|x|-R)dx\le \int_{\R^n} (\eta(u(t,x))+\varepsilon)\rho(N(t-t_0)+|x|-R)dx.
$$
Passing to the limit as $t\to 0$ in the above inequality with the help of initial condition (iv) of Definition~\ref{def1}, we arrive at the relation
\begin{align}\label{p7}
\int_{\R^n}\eta(u(t_0,x))\rho(|x|-R)dx\le \int_{\R^n}(\eta(u_0(x))+\varepsilon)\rho(|x|-Nt_0-R)dx\le \nonumber \\ \int_{\R^n}\eta(u_0(x))dx+\varepsilon \int_{\R^n}\rho(|x|-Nt_0-R)dx.
\end{align}
Observe that
\begin{align}\label{p8}
\int_{\R^n}\rho(|x|-Nt_0-R)dx\le \int_{|x|\le Nt_0+R+1}dx+e^{Nt_0+R}\int_{|x|>Nt_0+R+1} e^{-|x|}dx\le\nonumber\\
c_n(Nt_0+R+1)^n+nc_n e^{Nt_0+R}\int_{Nt_0+R+1}^{+\infty} e^{-r}r^{n-1}dr,
\end{align}
where $c_n$ is the measure of a unit ball in $\R^n$. Since
\begin{align*}
\int_{Nt_0+R+1}^{+\infty} e^{-r}r^{n-1}dr=\int_0^{+\infty}e^{-s-Nt_0-R-1}(s+Nt_0+R+1)^{n-1}ds\le \\ (Nt_0+R+1)^{n-1}e^{-Nt_0-R-1}\int_0^{+\infty}e^{-s}(1+s)^{n-1}ds=a(Nt_0+R+1)^{n-1}e^{-Nt_0-R-1},
\end{align*}
$a=\const$, it follows from (\ref{p8}) that for some constants $a_1,a_2$
$$
\varepsilon\int_{\R^n}\rho(|x|-N(\varepsilon)t_0-R)dx\le a_1\varepsilon(N(\varepsilon)t_0+R+1)^n\le a_2\varepsilon(1+\varepsilon^{-\frac{1}{m+1}})^n\mathop{\to}_{\varepsilon\to 0+}0
$$
(recall that $m+1>n$). Therefore, passing to the limit in (\ref{p7}) as $\varepsilon\to 0+$, we obtain that for all $t_0\in E$
\begin{equation}\label{p9}
\int_{\R^n}\eta(u(t_0,x))\rho(|x|-R)dx\le \int_{\R^n}\eta(u_0(x))dx.
\end{equation}
Now observe that $0\le\eta(u)\le u^+$ and $\eta(u)\to u^+$ as $\delta\to 0$. By Lebesgue dominated convergence theorem it follows from (\ref{p9}) in the limit as $\delta\to 0$ that for a.e. $t=t_0>0$
$$
\int_{\R^n}(u(t,x))^+\rho(|x|-R)dx\le \int_{\R^n}(u_0(x))^+dx<+\infty.
$$
By Fatou lemma this implies in the limit as $R\to\infty$ that
\begin{equation}\label{p10}
\int_{\R^n}(u(t,x))^+dx\le \int_{\R^n}(u_0(x))^+dx,
\end{equation}
that is exactly (\ref{L1}) with $k=0$.
To complete the proof for general $k\in\R$, we notice that $u-k$ is an e.s. of the problem
$$
u_t+\div_x(\varphi(u+k)-a(u+k)\nabla_x u)=0, \quad u(0,x)=u_0(x)-k.
$$
Applying (\ref{p10}) to this e.s., we obtain the desired estimate (\ref{L1}).
\end{proof}

\begin{corollary}\label{cor1a}
If $u=u(t,x)$ is an e.s. of (\ref{1}), (\ref{2}) then $\forall k\in\R$
\begin{equation}\label{L1a}
\int_{\R^n}(k-u(t,x))^+dx\le \int_{\R^n}(k-u_0(x))^+dx
\end{equation}
for a.e. $t>0$.
\end{corollary}

\begin{proof}
As is easy to verify, the function $v=-u$ is an e.s. of problem
\begin{equation}\label{red}
v_t+\div_x(-\varphi(-v)-a(-v)\nabla_x v)=0, \quad v(0,x)=-u_0(x).
\end{equation}
Applying (\ref{L1}) to this e.s. with $k$ replaced by $-k$, we arrive at (\ref{L1a}).
\end{proof}

\begin{corollary}\label{cor1b} Any e.s. $u=u(t,x)$ of (\ref{1}), (\ref{2}) satisfies the \textbf{maximum/minimum principle}
$$
a=\essinf u_0(x)\le u(t,x)\le b=\esssup u_0(x)  \ \mbox{ for a.e. } (t,x)\in\Pi.
$$
\end{corollary}

\begin{proof}
The maximum/minimum principle directly follows from (\ref{L1}) and (\ref{L1a}) with $k=b$ and $k=a$, respectively.
\end{proof}

\begin{lemma}\label{lem1}
Let $u=u(t,x)$ be an e.s. of problem (\ref{1}), (\ref{2}), $\|u\|_\infty\le M$. Then for each nonnegative function $f=f(t,x)\in C_0^\infty(\Pi)$
\begin{equation}
\int_\Pi \left(\sum_{j=1}^n (\div_x B_j(u))^2\right)f(t,x)dtdx \le C(f,M),
\end{equation}
where $C(f,M)$ is a constant depending only on $f$ and $M$.
\end{lemma}

\begin{proof}
we choose $\eta(u)=u^2/2$, $\psi(u)=T_{\eta'}(\varphi)(u)$, $H(u)=T_{\eta'}(A)(u)$.
By entropy relation (\ref{entr})
\begin{align*}
\int_\Pi \left(\sum_{j=1}^n (\div_x B_j(u))^2\right)f dtdx\le
\int_\Pi \bigl[\eta(u)f_t+\psi(u)\cdot\nabla_x f+H(u)\cdot D^2_x f\bigr]dtdx\le \\
C(f,M)\doteq \max_{|u|\le M}(\eta(u)+|\psi(u)|+|H(u)|)\int_\Pi\max(|f_t|,|\nabla_x f|,|D^2_x f|)dtdx,
\end{align*}
as was to be proved.
\end{proof}

\section{Main results}
As was established in \cite{ChPer1} (for the isotropic case see earlier paper \cite{Car}) by application of a variant of Kruzhkov's doubling variables method, the following Kato inequality holds for a pair of entropy solutions $u_1$, $u_2$.
\begin{equation}\label{Kato}
(|u_1-u_2|)_t+\div_x [\sgn(u_1-u_2)(\varphi(u_1)-\varphi(u_2))]-  D^2_x\cdot [\sgn(u_1-u_2)(A(u_1)-A(u_2))]\le 0 \ \mbox{ in } \D'(\Pi).
\end{equation}
Relation (\ref{Kato}) is the basis for the proof of comparison principle and uniqueness of e.s. But in the case under consideration when the flux functions are merely continuous while the diffusion matrix is degenerate these properties may be generally violated (see examples in \cite{KrPa1,KrPa2} for the case of conservation laws)
and additional conditions are necessary. Some of them can be found in \cite{AndIg,AndMal,MaT} in the isotropic case. The following result is the straightforward extension of \cite[Lemma~1]{ABK}.

\begin{lemma}\label{lem3}
Let $u_1=u_1(t,x)$, $u_2=u_2(t,x)$ be e.s. of problem (\ref{1}), (\ref{2}) with initial functions $u_{01}$, $u_{02}$, respectively. Suppose that for every $T>0$ the set
$$
A_T\doteq\{ \ (t,x)\in (0,T)\times\R^n \ | \ u_1(t,x)>u_2(t,x) \ \}
$$
has finite Lebesgue measure. Then for a.e. $t>0$
$$
\int_{\R^n}(u_1(t,x)-u_2(t,x))^+dx\le\int_{\R^n}(u_{01}(x)-u_{02}(x))^+dx.
$$
In particular, if $u_{01}\le u_{02}$ then $u_1\le u_2$ a.e. on $\Pi$ (comparison principle).
\end{lemma}

\begin{proof}
Putting (\ref{Kato}) together with the identity
$$
(u_1-u_2)_t+\div_x (\varphi(u_1)-\varphi(u_2))-D^2_x\cdot (A(u_1)-A(u_2))=0  \ \mbox{ in } \D'(\Pi),
$$
we obtain the relation
\begin{equation}\label{Kato+}
((u_1-u_2)^+)_t+\div_x[\theta(u_1-u_2)(\varphi(u_1)-\varphi(u_2))]-D^2_x\cdot[\theta(u_1-u_2)(A(u_1)-A(u_2))]\le 0  \ \mbox{ in } \D'(\Pi),
\end{equation}
where $\theta(u)=\sgn^+(u)$ is the Heaviside function.

We take $0<t_0<t_1$, and set $f=f(t,x)=(\theta_r(t-t_0)-\theta_r(t-t_1))p(x/l)$, where $r,l\in\N$, the nonnegative function $p(y)\in C_0^\infty(\R^n)$ is such that $0\le p(y)\le p(0)=1$, and the sequence $\theta_r(s)=\int_{-\infty}^s\omega_r(\sigma)d\sigma$ of approximations of the Heaviside function was constructed in Proposition~\ref{pro1} above. Applying (\ref{Kato+}) to the test function $f$, we obtain, after simple transforms
\begin{align}\label{11}
\int_\Pi (u_1(t,x)-u_2(t,x))^+\omega_r(t-t_1)p(x/l)dtdx\le \int_\Pi (u_1(t,x)-u_2(t,x))^+\omega_r(t-t_0)p(x/l)dtdx+ \nonumber\\
\frac{1}{l}\int_\Pi \theta(u_1-u_2)(\varphi(u_1)-\varphi(u_2))\cdot\nabla_y p(x/l)(\theta_r(t-t_0)-\theta_r(t-t_1))dtdx+ \nonumber\\
\frac{1}{l^2}\int_\Pi \theta(u_1-u_2)(A(u_1)-A(u_2))\cdot D^2_y p(x/l)(\theta_r(t-t_0)-\theta_r(t-t_1))dtdx.
\end{align}
Suppose that $t_0,t_1\in E$, where $E$ is a set of full measure of values $t$ such that $(t,x)$ is a Lebesgue point of the function $(u_1(t,x)-u_2(t,x))^+$ for a.e. $x\in\R^n$. Then $t_0,t_1$ are Lebesgue points of the functions $\displaystyle t\to \int_{\R^n} (u_1(t,x)-u_2(t,x))^+p(x/l)dx$, $l\in\N$, and it follows from (\ref{11}) in the limit as $r\to\infty$ that
\begin{align}\label{12}
\int_{\R^n} (u_1(t_1,x)-u_2(t_1,x))^+p(x/l)dx\le \int_{\R^n} (u_1(t_0,x)-u_2(t_0,x))^+p(x/l)dx+ \nonumber\\
\frac{1}{l}\int_{(t_0,t_1)\times\R^n} \theta(u_1-u_2)(\varphi(u_1)-\varphi(u_2))\cdot\nabla_y p(x/l)dtdx+ \nonumber\\
\frac{1}{l^2}\int_{(t_0,t_1)\times\R^n} \theta(u_1-u_2)(A(u_1)-A(u_2))\cdot D^2_y p(x/l)dtdx\le
\int_{\R^n} (u_1(t_0,x)-u_2(t_0,x))^+dx+ \nonumber\\ \left(\frac{1}{l}\|\varphi(u_1)-\varphi(u_2)\|_\infty\|\nabla_y p\|_\infty+\frac{1}{l^2}\|A(u_1)-A(u_2)\|_\infty\|D^2_y p\|_\infty\right)\int_{(0,t_1)\times\R^n} \theta(u_1-u_2)dtdx.
\end{align}
Observe that by our assumption $\int_{(0,t_1)\times\R^n} \theta(u_1-u_2)dtdx<+\infty$. Passing in (\ref{12}) to the limit as $E\ni t_0\to 0+$, we obtain that for all $t=t_1\in E$
\begin{align}\label{13}
\int_{\R^n} (u_1(t,x)-u_2(t,x))^+p(x/l)dx\le \int_{\R^n} (u_{01}(x)-u_{02}(x))^+p(x/l)dx+ \nonumber\\
\left(\frac{1}{l}\|\varphi(u_1)-\varphi(u_2)\|_\infty\|\nabla_y p\|_\infty+\frac{1}{l^2}\|A(u_1)-A(u_2)\|_\infty\|D^2_y p\|_\infty\right)\int_{(0,t)\times\R^n} \theta(u_1-u_2)dtdx,
\end{align}
where we use that
$$
|(u_1(t_0,x)-u_2(t_0,x))^+-(u_{01}(x)-u_{02}(x))^+|\le |u_1(t_0,x)-u_{01}(x)|+|u_2(t_0,x)-u_{02}(x)|\mathop{\to}_{E\ni t_0\to 0+} 0,
$$
by initial condition (iv).
Passing in (\ref{13}) to the limit as $l\to\infty$, with the help of Fatou's lemma, we arrive at
$$
\int_{\R^n}(u_1(t,x)-u_2(t,x))^+dx\le\int_{\R^n}(u_{01}(x)-u_{02}(x))^+dx.
$$
This completes the proof.
\end{proof}

Now, we are ready to establish existence of the largest and the smallest e.s. of our problem.

\subsection{Proof of Theorem~\ref{th2}}
We choose a strictly decreasing sequence $b_r$, $r\in\N$ such that $b_r>b=\esssup u_0(x)$ for all $r\in\N$,
and define a sequence of initial functions
$$
u_{0r}(x)=\left\{\begin{array}{lcr} u_0(x) & , & |x|\le r, \\ b_r & , & |x|>r. \end{array}\right.
$$
Let $u_r=u_r(t,x)$ be an e.s. of (\ref{1}), (\ref{2}) with initial data $u_{0r}$. Existence of this e.s. is known, for sufficiently regular nonlinearities it was established in \cite{ChPer1}, the general case can be treated by approximation with the help of known a priori estimates.
Observe that $\forall r\in\N$
$$
u_0(x)\le u_{0r+1}(x)\le u_{0r}(x)\le b_r \ \mbox{ a.e. on } \R^n, \mbox{ and }
\lim_{r\to\infty} u_{0r}(x)=u_0(x).
$$
Denote $d_r=b_r-b_{r+1}>0$. By the maximum principle $u_r\le b_r$ for all $r\in\N$. Therefore,
$$
\{(t,x) | u_{r+1}(t,x)>u_r(t,x)\}\subset\{(t,x) | b_{r+1}>u_r(t,x)\}=\{(t,x) | b_r-u_r(t,x)>d_r\}.
$$
By Chebyshev's inequality and Proposition~\ref{pro2} for each $T>0$
\begin{align*}
\meas\{ \ (t,x)\in (0,T)\times\R^n \ | \ u_{r+1}(t,x)>u_r(t,x) \ \}\le \\ \meas\{ \ (t,x)\in (0,T)\times\R^n \ | \  b_r-u_r(t,x)>d_r \ \}\le \\
\frac{1}{d_r}\int_{(0,T)\times\R^n}(b_r-u_r)^+dtdx\le  \frac{T}{d_r}\int_{\R^n}(b_r-u_{0r})^+dx= \frac{T}{d_r}\int_{|x|<r}(b_r-u_0)dx<+\infty.
\end{align*}
We see that the assumption of Lemma~\ref{lem3} regarded to the e.s. $u_{r+1}$ and $u_r$ is satisfied and by this lemma $u_{r+1}\le u_r$ a.e. on $\Pi$. Since
$u_{0r}\ge u_0\ge a\doteq\essinf u_0(x)$ then $u_r\ge a$, by the minimum principle. Hence, the sequence
$$
u_r(t,x)\mathop{\to}_{r\to\infty} u_+(t,x)\doteq\inf_{r>0} u_r(t,x)
$$
a.e. on $\Pi$, as well as in $L^1_{loc}(\Pi)$.
Further, $\|u_r\|_\infty\le M=\const$ and by Lemma~\ref{lem1} the sequences $\div_x B_j(u_r)$ are bounded in
$L^2_{loc}(\Pi)$ for all $j=1,\ldots,n$. Passing to a subsequence if necessary, we may suppose that
$\div_x B_j(u_r)\rightharpoonup p_j$ as $r\to\infty$ weakly in $L^2_{loc}(\Pi)$, $j=1,\ldots,n$. It follows from the relation
$$
\int_\Pi B_j(u_r)\cdot \nabla_x f dtdx=-\int_\Pi f\div_x B_j(u_r)dtdx, \quad f=f(t,x)\in C_0^\infty(\Pi)
$$
in the limit as $r\to\infty$ that
$$
\int_\Pi B_j(u_+)\cdot\nabla_x f dtdx=-\int_\Pi fp_j dtdx, \quad \forall f=f(t,x)\in C_0^\infty(\Pi),
$$
that is, $\div_x B_j(u_+)=p_j\in L^2_{loc}(\Pi,\R^n)$ in $\D'(\Pi)$ for all $j=1,\ldots,n$. We conclude that $u_+$ satisfies assumption (i) of Definition~\ref{def1}. Further, in view of (ii) applied to e.s. $u_r$ we have
for every $g(u)\in C^1(\R)$, $j=1,\ldots,n$, $f\in C_0^\infty(\Pi)$
$$
\int_\Pi T_g(B_j)(u_r)\cdot\nabla_x f dtdx=-\int_\Pi g(u_r)\div_x B_j(u_r)fdtdx.
$$
In the limit as $r\to\infty$ this implies that $\forall f\in C_0^\infty(\Pi)$
$$
\int_\Pi T_g(B_j)(u_+)\cdot\nabla_x f dtdx=-\int_\Pi g(u_+)p_jfdtdx,
$$
that is,
$$
\div_x T_g(B_j)(u_+)=g(u_+)p_j=g(u_+)\div_x B_j(u_+) \ \mbox{ in } \D'(\Pi).
$$
We find that $u_+$ satisfies the chain rule (ii).

By Proposition~\ref{pro1} e.s. $u_r$ satisfies integral inequality (\ref{eint}): for any convex entropy $\eta\in C^2(\R)$ and each nonnegative test function $f=f(t,x)\in C_0^\infty(\bar\Pi)$
\begin{align}\label{eint2}
\int_\Pi \bigl[\eta(u_r)f_t+T_{\eta'}(\varphi)(u_r)\cdot\nabla_x f+T_{\eta'}(A)(u_r)\cdot D^2_x f-\nonumber\\ f\eta''(u_r)\sum_{j=1}^n (\div_x B_j(u_r))^2\bigr]dtdx+ \int_{\R^n} \eta(u_0(x))f(0,x)dx\ge 0.
\end{align}
Observe that for each $j=1,\ldots,n$
$$
v_{jr}\doteq\sqrt{\eta''(u_r)}\div_x B_j(u_r)\rightharpoonup v_j\doteq\sqrt{\eta''(u_+)}\div_x B_j(u_+)
$$
as $r\to\infty$ weakly in $L^2_{loc}(\Pi)$. By the known property of weak convergence
$$
\int_\Pi f\sum_{j=1}^n v_j^2dtdx\le\liminf_{r\to\infty}\int_\Pi f\sum_{j=1}^n v_{jr}^2dtdx,
$$
which reduces to the inequality
\begin{equation}\label{w1}
\int_\Pi f\eta''(u_+)\sum_{j=1}^n (\div_x B_j(u_+))^2dtdx\le\liminf_{r\to\infty}\int_\Pi f\eta''(u_r)\sum_{j=1}^n (\div_x B_j(u_r))^2dtdx.
\end{equation}
Passing in (\ref{eint2}) to the limit as $r\to\infty$ and taking into account (\ref{w1}), we obtain that
for every $f=f(t,x)\in C_0^\infty(\bar\Pi)$, $f\ge 0$,
\begin{align*}
\int_\Pi \bigl[\eta(u_+)f_t+T_{\eta'}(\varphi)(u_+)\cdot\nabla_x f+T_{\eta'}(A)(u_+)\cdot D^2_x f-\nonumber\\ f\eta''(u_+)\sum_{j=1}^n (\div_x B_j(u_+))^2\bigr]dtdx+ \int_{\R^n} \eta(u_0(x))f(0,x)dx\ge 0.
\end{align*}
Thus, the function $u_+$ also satisfies (\ref{eint}). By Proposition~\ref{pro1}.
$u_+$ is an e.s. of (\ref{1}), (\ref{2}).

Let us demonstrate that $u_+$ is the largest e.s. of this problem.
For that, we choose an arbitrary e.s. $u=u(t,x)$ of (\ref{1}), (\ref{2}). By the maximum principle,
$u\le b$. Therefore, in the set $\Pi_T=(0,T)\times\R^n$ $\{u>u_r\}\subset\{b>u_r\}= \{b_r-u_r>b_r-b\}$ and consequently
\begin{align*}
\meas\{u>u_r\}\le\frac{1}{b_r-b}\int_{\Pi_T}(b_r-u_r)^+dx\le \frac{T}{b_r-b}\int_{|x|<r}(b_r-u_0)dx<+\infty,
\end{align*}
where we use again Chebyshev's inequality and Proposition~\ref{pro2}.
Hence, the requirement of Lemma~\ref{lem3}, applied to the e.s. $u$ and $u_r$, is satisfied and, by the comparison principle,
the inequality $u_0\le u_{0r}$ implies that $u\le u_r$ a.e. on $\Pi$. In the limit as $r\to\infty$ we conclude that $u\le u_+$ a.e. on $\Pi$. Hence, $u_+$ is the unique largest e.s.

Let $v_+=v_+(t,x)$ be the largest e.s. of (\ref{red}). It is clear that the function
$u_-(t,x)=-v_+(t,x)$ is the smallest e.s. of the original problem. The inequality $u_-\le u_+$ is evident.
The proof of Theorem~\ref{th2} is complete.

\subsection{Proof of Theorem~\ref{th2contr}}
Let $u_{1+}$, $u_{2+}$ be the largest e.s. of (\ref{1}), (\ref{2}) with initial data $u_{10}$, $u_{20}$. We choose a strictly decreasing sequence $b_r$, $r\in\N$, such that $b_r>\max(\|u_{10}\|_\infty,\|u_{20}\|_\infty)$, and consider the sequences
$$
u_{1r}^0(x)=\left\{\begin{array}{lcr} u_{10}(x) & , & |x|\le r, \\ b_r & , & |x|>r \end{array}\right., \quad
u_{2r}^0(x)=\left\{\begin{array}{lcr} u_{20}(x) & , & |x|\le r, \\ b_r+1 & , & |x|>r
\end{array}\right..
$$
Then the corresponding sequences $u_{1r},u_{2r}$ of e.s. strongly converges as $r\to\infty$ to the largest e.s. $u_{1+}$, $u_{2+}$, respectively. By the maximum principle $u_{1r}\le b_r$. Therefore, for each $T>0$
\begin{align*}
\{(t,x)\in\Pi_T | u_{1r}(t,x)>u_{2r}(t,x)\}\subset \\ \{(t,x)\in\Pi_T | b_r>u_{2r}(t,x)\}=\{(t,x)\in\Pi_T | b_r+1-u_{2r}(t,x)>1\},
\end{align*}
and by Chebyshev's inequality and Proposition~\ref{pro2}
\begin{align*}
\meas\{(t,x)\in\Pi_T | u_{1r}(t,x)>u_{2r}(t,x)\}\le\int_{\Pi_T} (b_r+1-u_{2r}(t,x))dtdx\le \\ T\int_{|x|<r} (b_r+1-u_{20}(x))dx<\infty.
\end{align*}
By Lemma~\ref{lem3} again we conclude that for a.e. $t>0$
\begin{align*}
\int_{\R^n}(u_{1r}(t,x)-u_{2r}(t,x))^+dx\le\int_{\R^n}(u_{1r}^0(x)-u_{2r}^0(x))^+dx=\\ \int_{|x|<r}(u_{10}(x)-u_{20}(x))^+dx\le
\int_{\R^n}(u_{10}(x)-u_{20}(x))^+dx.
\end{align*}
Passing to the limit as $r\to\infty$ with the help of Fatou's lemma, we derive the desired result
$$
\int_{\R^n}(u_{1+}(t,x)-u_{2+}(t,x))^+dx\le \int_{\R^n}(u_{10}(x)-u_{20}(x))^+dx.
$$
Concerning the smallest e.s., we observe that $u_{1-}=-v_{1+}$, $u_{2-}=-v_{2+}$, where $v_{1+}$, $v_{2+}$ are the largest e.s. of problem (\ref{red}) with initial data $-u_{10}(x)$, $-u_{20}(x)$, respectively. As was already proved,
$$
\int_{\R^n}(v_{2+}(t,x)-v_{1+}(t,x))^+dx\le \int_{\R^n}(u_{10}(x)-u_{20}(x))^+dx,
$$
which reduces to the desired relation
$$
\int_{\R^n}(u_{1-}(t,x)-u_{2-}(t,x))^+dx\le \int_{\R^n}(u_{10}(x)-u_{20}(x))^+dx.
$$

\subsection{The case of periodic initial data}
Now we assume that the initial function $u_0(x)$ is periodic with a lattice of periods $L\subset\R^n$. Hence, $u_0(x+e)=u_0(x)$ a.e. on $\R^n$ for each $e\in L$.

\begin{theorem}\label{th3}
The largest e.s. $u_+$ and the smallest e.s. $u_-$ of the problem (\ref{1}), (\ref{2}) are space-periodic and coincide: $u_+=u_-$.
\end{theorem}

\begin{proof}
Let $e\in L$. In view of periodicity of the initial function it is obvious that $u(t,x+e)$ is an e.s. of (\ref{1}), (\ref{2}) if and only if $u(t,x)$ is an e.s. of the same problem. Therefore, $u_+(t,x+e)$ is the largest e.s.
of (\ref{1}), (\ref{2}) together with $u_+$. By the uniqueness $u_+(t,x+e)=u_+(t,x)$ a.e. on $\Pi$ for all $e\in L$, that is $u_+$ is a space periodic function. In the same way we prove space periodicity of the smallest e.s. $u_-$.
Since $u_\pm$ are weak solutions of (\ref{1}), we find
\begin{equation}\label{14}
(u_+-u_-)_t+\div_x(\varphi(u_+)-\varphi(u_-))-D^2_x\cdot(A(u_+)-A(u_-))=0 \ \mbox{ in } \D'(\Pi).
\end{equation}
Let $\alpha(t)\in C_0^1(\R_+)$, $\beta(y)\in C_0^2(\R^n)$, $\alpha(t),\beta(y)\ge 0$, $\displaystyle\int_{\R^n}\beta(y)dy=1$. Applying (\ref{14}) to the test function $\alpha(t)\beta(x/k)$, with $k\in\N$, we arrive at the relation
\begin{align*}
\int_\Pi(u_+-u_-)\alpha'(t)\beta(x/k)dtdx+k^{-1}\int_\Pi(\varphi(u_+)-\varphi(u_-))\cdot\nabla_y\beta(x/k)\alpha(t)dtdx+\\
k^{-2}\int_\Pi(A(u_+)-A(u_-))\cdot D^2_y\beta(x/k)\alpha(t)dtdx=0.
\end{align*}
Multiplying this equality by $k^{-n}$ and passing to the limit as $k\to\infty$, we obtain
\begin{equation}\label{15}
\int_{\R_+\times\T^n}(u_+(t,x)-u_-(t,x))\alpha'(t)dtdx=0,
\end{equation}
where $\T^n=\R^n/L$ is a torus (which can be identified with a fundamental parallelepiped), equipped with the normalized Lebesgue measure $dx$. We use here the known property
$$
\lim_{k\to\infty}k^{-n} \int_\Pi\mu(t,x)\alpha(t)\beta(x/k)dtdx=\int_{\R_+\times\T^n}\alpha(t)\mu(t,x)dtdx
$$
for an arbitrary $x$-periodic function $\mu(t,x)\in L^1_{loc}(\Pi)$. Identity (\ref{15}) means that
$$
\frac{d}{dt}\int_{\T^n}(u_+(t,x)-u_-(t,x))dx=0 \ \mbox{ in } \D'(\R_+).
$$
This implies that for a.e. $t,t_0$, $t>t_0$
\begin{equation}\label{16}
\int_{\T^n}(u_+(t,x)-u_-(t,x))dx=\int_{\T^n}(u_+(t_0,x)-u_-(t_0,x))dx.
\end{equation}
Taking into account the initial relation (iv), we find that
$$
\int_{\T^n}(u_+(t_0,x)-u_-(t_0,x))dx\le\int_{\T^n} |u_+(t_0,x)-u_0(x)|dx+\int_{\T^n} |u_-(t_0,x)-u_0(x)|dx\to 0
$$
as $t_0\to 0$ running over a set of full measure. Therefore (\ref{16}) implies in the limit as $t_0\to 0$ that
$$\int_{\T^n}(u_+(t,x)-u_-(t,x))dx=0$$ for a.e. $t>0$. Since $u_+\ge u_-$, we conclude that $u_+=u_-$ a.e. on $\Pi$.
\end{proof}
Since any e.s. of (\ref{1}), (\ref{2}) is situated between $u_-$ and $u_+$, we deduce the following

\begin{corollary}\label{cor4}
An e.s. of (\ref{1}), (\ref{2}) is unique and coincides with $u_+$.
\end{corollary}

More generally we establish below the comparison principle formulated in Theorem~\ref{th4}.
\subsection{Proof of Theorem~\ref{th4}}

For definiteness suppose that the function $u_0(x)$ is periodic. The case of periodic $v_0$ is treated similarly. By Theorem~\ref{th3}  the  functions $u_+=u_-$ coincide with the unique e.s. of (\ref{1}), (\ref{2}).
As follows from Theorem~\ref{th2contr} and the condition $u_0\le v_0$, $u_-\le v_-$. Therefore, $u=u_-\le v_-\le v$, as was to be proved.

\medskip
We underline that for conservation laws (\ref{con}) Theorems~\ref{th2}--\ref{th4} were establishes in \cite{PaMax1,PaMax2,PaIzv}.
Moreover, the comparison principle and the uniqueness of e.s. remain valid in the case when the initial function is periodic at least in $n-1$ independent directions, this can be proved by the same methods as for conservations laws, see \cite{PaMax2,PaIzv}.

\section{The long time decay property}\label{secDec}

We will essentially rely on results about decay of space-periodic e.s. Assume that the initial function $u_0$ is periodic. Let
$$I=\int_{\T^n} u_0(x)dx$$ be the mean value of initial data, and
$$
L'=\{\xi\in\R^n | \xi\cdot e\in\Z \ \forall e\in L \}
$$
be the dual lattice to the lattice of periods $L$.

\begin{theorem}\label{thM}
Assume that for all $\xi\in L'$, $\xi\not=0$ there is no vicinity of $I$, where simultaneously the function $\xi\cdot\varphi(u)$ is affine and the function $a(u)\xi\cdot\xi=0$ almost everywhere. Then
\begin{equation}\label{decp}
\esslim_{t\to+\infty} u(t,x)=I \ \mbox{ in } L^1(\T^n).
\end{equation}
\end{theorem}

Theorem~\ref{thM} was established in recent paper \cite{PaJDE} and generalizes earlier results \cite{ChPer2,PaNHM}.

To prove Theorem~\ref{thD}, we will follow the scheme of paper \cite{PaSIMA}, where the case of conservation laws $a\equiv 0$ was treated.

\subsection{Auxiliary lemmas}
The following simple lemmas were proved in \cite{PaSIMA}. For the sake of completeness we reproduce them with the proofs.

\begin{lemma}\label{lemA}
The norms $\|\cdot\|_V$ defined in (\ref{normV}) are mutually equivalent.
\end{lemma}

\begin{proof}
Let $V_1,V_2$ be open bounded sets in $\R^n$, and $K_1=\Cl V_1$ be the closure of $V_1$. Then $K_1$ is a compact set while $y+V_2$, $y\in K_1$, is its open covering.
By the compactness there is a finite set $y_i$, $i=1,\ldots,m$, such that
$\displaystyle K_1\subset \bigcup\limits_{i=1}^m (y_i+V_2)$. This implies that for every $y\in\R^n$ and $u=u(x)\in L^\infty(\R^n)$
$$
\int_{y+V_1}|u(x)|dx\le\sum_{i=1}^m \int_{y+y_i+V_2}|u(x)|dx\le m\|u\|_{V_2}.
$$
Hence, $\forall u=u(x)\in L^\infty(\R^n)$
$$
\|u\|_{V_1}=\sup_{y\in\R^n}\int_{y+V_1}|u(x)|dx\le m\|u\|_{V_2}.
$$
Changing the places of $V_1$, $V_2$, we obtain the inverse inequality
$\|u\|_{V_2}\le l\|u\|_{V_1}$ for all $u\in L^\infty(\R^n)$, where $l$ is some positive constant. This completes the proof.
\end{proof}

\begin{lemma}\label{lemB}
Let $X_\alpha$, $\alpha\in\N$, be a countable family of proper linear subspaces of $\R^n$. Then there exists a lattice $L\subset\R^n$ such that $L\cap X_\alpha=\{0\}$ for all $\alpha\in\N$.
\end{lemma}

\begin{proof}
Denote by $\mathrm{L}_n$ the linear space of linear endomorphisms of $\R^n$, and by
$\mathrm{GL}_n$ the group of linear automorphisms of $\R^n$. It is clear that $\mathrm{GL}_n$ is an open subset of $\mathrm{L}_n$, this set can be identified with the set of all $n\times n$ matrices with nonzero determinant. For $\alpha\in\N$, $\xi\in\Z^n$, $\xi\not=0$, we define the sets
$$
H_{\xi,\alpha}=\{ \ A\in\mathrm{L}_n: \ A\xi\in X_\alpha \ \}, \quad
H=\bigcup_{\xi\in\Z^n\setminus\{0\},\alpha\in\N} H_{\xi,\alpha}.
$$
Obviously, the sets $H_{\xi,\alpha}$ are proper linear subspaces of $\mathrm{L}_n$ and therefore they have zero Lebesgue measure in $\mathrm{L}_n$. This implies that
$H$ is a set of zero measure as a countable union of $H_{\xi,\alpha}$. Since the measure of $\mathrm{GL}_n$ is positive (even infinite), then $\mathrm{GL}_n\not\subset H$ and we can find $A\in \mathrm{GL}_n$ such that $A\not\in H$. We define the lattice $L$ as the image of the standard lattice $\Z^n$ under the automorphism $A$: $L=A(\Z^n)$. Since $A\notin H$, we conclude that $L$ satisfies the required condition $L\cap X_\alpha=\{0\}$ $\forall\alpha\in\N$.
\end{proof}

We define the set $F\subset\R$ consisting of points $u_0$ such that in any neighborhood of $u_0$ it is not possible that simultaneously the vector $\varphi(u)$ is affine and the matrix $a(u)=0$ on this neighborhood. We also denote $F_+=F\cap (0,+\infty)$, $F_-=F\cap (-\infty,0)$.

\begin{lemma}\label{lemC}
Assume that nonlinearity-diffusivity assumption (\ref{GN}) is satisfied. Then
\begin{equation}\label{Fpm}
\sup F_-=\inf F_+=0.
\end{equation}
\end{lemma}

\begin{proof}
Supposing the contrary, we find that either $\sup F_-<0$ or $\inf F_+>0$. We consider the latter case $\inf F_+>0$, the former case $\sup F_-<0$ is treated similarly.
Let $0<b<\inf F_+$ (notice that $\inf F_+=+\infty$ in the case $F_+=\emptyset$). We see that $(0,b)\cap F=\emptyset$, that is, the vector $\varphi(u)$ is affine and the matrix $a(u)=0$ in some vicinity of each point in $(0,b)$. We deduce that, firstly,
$a(u)=0$ on $(0,b)$. Secondly, $\varphi(u)\in C^\infty((0,b))$ and $\varphi'(u)$ is a piecewise constant continuous function on $(0,b)$. This is possible only if $\varphi'(u)$ is constant, $\varphi'(u)\equiv c\in\R^n$. This implies that $\varphi(u)=uc+d$ on $(0,b)$, $d\in\R^n$. Hence, the vector $\varphi(u)$ is affine on $(0,b)$ while the matrix $a(u)=0$ on  $(0,b)$, which contradicts to (\ref{GN}).
\end{proof}

\subsection{Proof of Theorem~\ref{thD}}

The proof is relied on the decay property for periodic e.s. First we choose a lattice of periods $L$.

Let $J$ be the sets of intervals $I=(a,b)$ with rational ends $a,b\in\Q$ such that $I\cap F\not=\emptyset$. It is clear that $J$ is a countable set.
For each $I\in J$ we define the linear sets
\begin{align*}
X_I=\{ \ \xi\in\R^n: \ \mbox{ the function } u\to\xi\cdot\varphi(u) \mbox{ is affine on } I, \\
\mbox{ and the function } u\to a(u)\xi\cdot\xi=0 \mbox{ on } I \ \}.
\end{align*}
Then $X_I\not=\R^n$, otherwise, the vector $\varphi(u)$ is affine on $I$ while the matrix $a(u)=0$ on $I$, which contradicts to the condition that $I$ is a neighborhood of some point $u_0\in I\cap F$. Hence $X_I$, $I\in J$, are proper linear subspaces of $\R^n$. By Lemma~\ref{lemB} we can find a lattice $L_1$ in $\R^n$
such that $\xi\notin X_I$ for all $\xi\in L_1$, $\xi\not=0$, and all $I\in J$. Let $L=L_1'=\{e\in\R^n: \xi\cdot e\in\Z \ \forall\xi\in L_1\}$ be the dual lattice. Then by the duality $L_1=L'$. By the density of $\Q$, any nonempty interval $(a,b)$ intersecting with F contains some interval $I\in J$. Since every nonzero $\xi\in L'=L_1$ does not belong to $X_I$, we claim that either the function $\xi\cdot\varphi(u)$ is not affine on $I$ or $a(u)\xi\cdot\xi\not\equiv 0$ on $I$. All  the more this property holds on the larger interval $(a,b)$. Hence,
\begin{eqnarray}\label{p1}
\forall\xi\in L', \xi\not=0, \mbox{ either the function } u\to\xi\cdot\varphi(u) \nonumber\\
\mbox{ is not affine or } a(u)\xi\cdot\xi\not\equiv 0 \mbox{ on intervals intersecting with } F.
\end{eqnarray}
Let $e_k$, $k=1,\ldots,n$, be a basis of the lattice $L$. We define for $r>0$ the parallelepiped
$$
P_r=\left\{ \ x=\sum_{k=1}^n x_ke_k: \ -r/2\le x_k<r/2, k=1,\ldots,n \ \right\}.
$$
It is clear that $P_r$ is a fundamental parallelepiped for a lattice $rL$.
We introduce the functions
$$
v_{0r}^+(x)=\sup_{e\in L} u_0(x+re), \quad v_{0r}^-(x)=\inf_{e\in L} u_0(x+re).
$$
Since $L$ is countable, these functions are well-defined in $L^\infty(\R^n)$,
and $\|v_{0r}^\pm\|_\infty\le C_0\doteq\|u_0\|_\infty$. It is clear that $v_{0r}^\pm(x)$ are $rL$-periodic and
\begin{equation}\label{ine2v}
v_{0r}^-(x)\le u_0(x)\le v_{0r}^+(x).
\end{equation}
We denote
$$
V_r(x)=\sup_{e\in L} |u_0(x+re)|, \quad M_r=\frac{1}{|P_r|}\int_{P_r} V_r(x)dx.
$$
It is clear that for a.e. $x\in\R^n$
$$
|v_{0r}^\pm(x)|\le V_r(x)\le C_0.
$$
As was demonstrated in \cite{PaSIMA}, under condition (\ref{van})
\begin{equation}\label{l3}
M_r\to 0 \ \mbox{ as } r\to+\infty.
\end{equation}
For the sake of completeness we give details of the proof.

We fix $\varepsilon>0$ and define the set $A=\{ \ x\in\R^n: \ |u_0(x)|>\varepsilon \ \}$.
In view of (\ref{van}) the measure of this set is finite, $\meas A=p<+\infty$. We also define the sets
$$
A_r^e=\{ \ x\in P_r: \ x+re\in A \ \}\subset P_r, \quad r>0, \ e\in L, \quad A_r=\bigcup_{e\in L} A_r^e.
$$
By the translation invariance of Lebesgue measure and the fact that $\R^n$ is the disjoint union of the sets $re+P_r$, $e\in L$, we have
$$
\sum_{e\in L}\meas A_r^e=\sum_{e\in L}\meas (re+A_r^e)=\sum_{e\in L}\meas (A\cap (re+P_r))=\meas A=p.
$$
This implies that
\begin{equation}\label{Ar}
\meas A_r\le\sum_{e\in L}\meas A_r^e=p.
\end{equation}
If $x\in P_r\setminus A_r$ then $|u_0(x+re)|\le\varepsilon$ for all $e\in L$, which implies that $V_r(x)\le\varepsilon$. Taking (\ref{Ar}) into account, we find
$$
\int_{P_r} V_r(x)dx=\int_{A_r} V_r(x)dx+\int_{P_r\setminus A_r} V_r(x)dx\le C_0\meas A_r+\varepsilon\meas P_r\le C_0p+\varepsilon|P_r|.
$$
It follows from this estimate that
$$\limsup_{r\to+\infty} M_r\le\lim_{r\to+\infty}\left(\frac{C_0p}{|P_r|}+\varepsilon\right)=\varepsilon$$ and since $\varepsilon>0$ is arbitrary, we conclude that (\ref{l3}) holds.
Let
$$
M_r^\pm=\frac{1}{|P_r|}\int_{P_r} v_{0r}^\pm(x)dx
$$
be mean values of $rL$-periodic functions $v_{0r}^\pm(x)$.
In view of (\ref{l3})
\begin{equation}\label{l4}
|M_r^\pm|\le M_r\mathop{\to}_{r\to\infty} 0.
\end{equation}
By (\ref{l4}) and (\ref{Fpm}) we can find such values $B_r^-,B_r^+\in F$, where $r>0$ is sufficiently large, that $B_r^-\le M_r^-\le M_r^+\le B_r^+$, and that $B_r^\pm\to 0$ as $r\to\infty$. We define the $rL$-periodic functions
$$
u_{0r}^+(x)=v_{0r}^+(x)-M_r^++B_r^+\ge v_{0r}^+(x), \quad u_{0r}^-(x)=v_{0r}^-(x)-M_r^-+B_r^-\le v_{0r}^-(x)
$$
with the mean values $B_r^+,B_r^-$, respectively. In view of (\ref{ine2v}), we have
\begin{equation}\label{ine2}
u_{0r}^-(x)\le u_0(x)\le u_{0r}^+(x).
\end{equation}
Let $u_r^\pm$ be unique (by Corollary~\ref{cor4}) e.s. of (\ref{1}), (\ref{2}) with initial functions
$u_{0r}^\pm$, respectively. Taking into account that $(rL)'=\frac{1}{r}L'$, we derive from (\ref{p1}) that the requirement of Theorem~\ref{thM}, corresponding to the lattice $rL$ and the mean value $M=B_r^\pm\in F$, is satisfied. By this theorem we find that
\begin{equation}\label{l5}
\lim_{t\to+\infty}\int_{P_r}|u_r^\pm(t,x)-B_r^\pm|dx=0.
\end{equation}
By the periodicity, for each $y\in\R^n$
$$
\int_{y+P_r}|u_r^\pm(t,x)-B_r^\pm|dx=\int_{P_r}|u_r^\pm(t,x)-B_r^\pm|dx,
$$
which readily implies that for $V=\Int P_r$
$$
\|u_r^\pm(t,x)-B_r^\pm\|_V=\int_{P_r}|u_r^\pm(t,x)-B_r^\pm|dx.
$$
In view of Lemma~\ref{lemA} we have the estimate
$$\|u_r^\pm(t,x)-B_r^\pm\|_X\le C\int_{P_r}|u_r^\pm(t,x)-B_r^\pm|dx, \ C=C_r=\const.$$
By (\ref{l5}) we claim that
\begin{equation}\label{l5a}
\lim_{t\to+\infty}\|u_r^\pm(t,\cdot)-B_r^\pm\|_X=0.
\end{equation}
Let $u=u(t,x)$ be an e.s. of the original problem (\ref{1}), (\ref{2}) with initial data $u_0(x)$. Since the functions $u_{0r}^\pm$ are periodic, then it follows from (\ref{ine2}) and the comparison principle stated in Theorem~\ref{th4} that $u_r^-\le u\le u_r^+$ a.e. in $\Pi$. This readily implies the relation
\begin{eqnarray}\label{l6}
\|u(t,\cdot)\|_X\le \|u_r^-(t,\cdot)\|_X+\|u_r^+(t,\cdot)\|_X\le \nonumber\\
\|u_r^-(t,x)-B_r^-\|_X+\|u_r^+(t,x)-B_r^+\|_X+c(|B_r^-|+|B_r^+|),
\end{eqnarray}
where $c$ is Lebesgue measure of the unit ball $|x|<1$ in $\R^n$.
In view of (\ref{l5a}) it follows from (\ref{l6}) in the limit as $t\to+\infty$
that
$$
\limsup_{t\to+\infty}\|u(t,\cdot)\|_X\le c(|B_r^-|+|B_r^+|).
$$
Since $B_r^\pm\to 0$ as $r\to\infty$, the latter relation implies the desired decay property
$$
\lim_{t\to+\infty}\|u(t,\cdot)\|_X=0.
$$

\section*{Acknowledgements}
This work was supported by the ``RUDN University Program 5-100'', by the Ministry of Science and Higher Education of the Russian  Federation (project no. 1.445.2016/1.4) and by the Russian Foundation for Basic Research (grant 18-01-00258-a.)


\begin{thebibliography}{20}
\bibitem{ABK}
B.\,P.~Andreianov, Ph.~B\'enilan, S.N.~Kruzhkov,
$L^1$-Theory of Scalar Conservation Law with Continuous Flux Function, J. Funct. Anal. 171:1 (2000), 15--33.
\bibitem{AndIg}
B.\,P.~Andreianov, N.~Igbida, On uniqueness techniques for degenerate convection--diffusion
problems, Int. J. Dynamical Systems and Differential Equations, 4:1–2 (2012), 3--34.
\bibitem{AndMal}
B.~Andreianov, M.~Maliki, A note on uniqueness of entropy solutions to degenerate
parabolic equations in $\R^N$, NoDEA: Nonlin. Diff. Eq. Appl., 17:1 (2010), 109--118.
\bibitem{ChPer1}
G.-Q.~Chen, B.~Perthame, Well-posedness for non-isotropic degenerate parabolic-hyperbolic
equations, Ann. Inst. H. Poincar\'e Anal. Non Lin\'eaire 20 (2003), 645--668.
\bibitem{ChPer2}
G.-Q.~Chen, B.~Perthame, Large-time behavior of periodic entropy solutions to anisotropic
degenerate parabolic-hyperbolic equations, Proc. American Math. Soc. 137:9 (2009), 3003--3011.
\bibitem{BK}
Ph.~B\'enilan, S.N.~Kruzhkov, Conservation laws with continuous flux function,
NoDEA: Nonlin. Diff. Eq. Appl., 3 (1996), 395--419.
\bibitem{Car}
J.~Carrillo, Entropy solutions for nonlinear degenerate problems, Arch. Ration. Mech. Anal.,
147 (1999), 269--361.
\bibitem{Kr}
S.\,N.~Kruzhkov,
First order quasilinear equations in several independent variables, Mat. Sb. (N.S.), 81 (1970), 228--255.
\bibitem{KrPa1}
S.\,N.~Kruzhkov, E.\,Yu.~Panov, First-order conservative quasilinear laws with an infinite domain of dependence on the initial data, Soviet Math. Dokl., 42 (1991), 316--321.
\bibitem{KrPa2}
S.\,N.~Kruzhkov, E.\,Yu.~Panov, Osgood's type conditions for uniqueness of entropy solutions to Cauchy problem for quasilinear conservation laws of the first order, Ann. Univ. Ferrara Sez. VII (N.S.), 40 (1994), 31--54.
\bibitem{MaT}
M.~Maliki, H. Tour\'e, Uniqueness of entropy solutions for nonlinear degenerate parabolic problem. J. Evol. Equ. 3:4 (2003), 603--622.
\bibitem{PaMax1}
E.\,Yu.~Panov, A remark on the  theory of generalized entropy sub- and supersolutions of the Cauchy problem for a first-order quasilinear equation, Differ. Equ., 37 (2001), 272--280.
\bibitem{PaMax2}
E.\,Yu.~Panov, Maximum and minimum generalized entropy solutions to the Cauchy problem for a first-order quasilinear equation, Sb. Math., 193:5 (2002), 727--743.
\bibitem{PaIzv}
E.\,Yu.~Panov, On generalized entropy solutions of the Cauchy problem for a first order quasilinear equation in the class of locally summable functions, Izv. Math., 66:6 (2002), 1171--1218.
\bibitem{PaNHM}
E.\,Yu.~Panov, On a condition of strong precompactness and the decay of periodic entropy
solutions to scalar conservation laws, Netw. Heterog. Media 11:2 (2016), 349--367.
\bibitem{PaJHDE}
E.\,Yu.~Panov,
On the Cauchy problem for scalar conservation laws in the class of Besicovitch almost periodic functions: Global well-posedness and decay property,
J. Hyperbolic Differ. Equ., 13 (2016), 633--659.
\bibitem{PaJDE}
E.\,Yu.~Panov, Decay of periodic entropy solutions to degenerate nonlinear parabolic equations, J. Differential Equations, 269:1 (2020),  862--891.
\bibitem{PaSIMA}
E.\,Yu.~Panov, On Decay of Entropy Solutions to Multidimensional Conservation Laws, SIAM J. Math. Anal., 52:2 (2020), 1310--1317.
\end{thebibliography}
\end{document}